\documentclass{amsart}

\usepackage{amsmath}
\usepackage{amssymb}
\usepackage{graphicx}
\usepackage{amsthm}
\usepackage{setspace}
\usepackage{listofitems}
\usepackage{hyperref}
\usepackage[T1]{fontenc}
\usepackage{lmodern}
\usepackage{tikz}

\usetikzlibrary{calc}
\usetikzlibrary{arrows.meta} 

\tikzset{
  vertex/.style={circle,draw=black,fill=black!12,inner sep=2pt,minimum size=20pt,font=\small},
  edge/.style={thick}
}

\newcounter{vidx}
\newcommand{\vlabel}{\stepcounter{vidx}$v_{\thevidx}$}

\DeclareMathOperator{\sgn}{sgn}
\DeclareMathOperator{\Real}{Re}

\newcommand{\R}{\mathbb{R}}
\newcommand{\C}{\mathbb{C}}
\newcommand{\AS}{Atiyah\textendash Sutcliffe }
\newcommand{\norm}[1]{\lVert #1 \rVert}

\title{Finite graphs and configurations of points}
\author{Joseph Malkoun}
\address{Department of Mathematics and Statistics\\
Faculty of Natural and Applied Sciences\\
Notre Dame University-Louaize, Zouk Mosbeh\\
P.O.Box: 72, Zouk Mikael,\\
Lebanon}
\email{joseph.malkoun@ndu.edu.lb}

\date{\today}

\newtheorem{theorem}{Theorem}[section]
\newtheorem{corollary}[theorem]{Corollary}
\newtheorem{conjecture}{Conjecture}[part]

\newtheorem{lemma}[theorem]{Lemma}
\newtheorem{proposition}[theorem]{Proposition}

\theoremstyle{remark}

\theoremstyle{definition}
\newtheorem{definition}[theorem]{Definition}

\begin{document}

\begin{abstract} We generalize the Atiyah problem on configurations and the related Atiyah\textendash Sutcliffe conjectures 1 and 2 using finite graphs, configurations of points and tensors. Our conjectures are intriguing geometric inequalities, defined using the pairwise directions of the configuration of points, just as in the original problem. The generalization of the Atiyah determinant to our setting is no longer a determinant. We call it the $G$-amplitude function, where $G$ is a finite simple graph, in analogy with probability amplitudes in quantum physics. If $G = K_n$ is the complete graph with $n$ vertices, we recover the Atiyah\textendash Sutcliffe conjectures 1 and 2.
\end{abstract}

\maketitle

\section{Introduction}

The Atiyah problem on configurations of points saw the light of day towards the beginning of this century, cf. \cite{Ati-2000}, \cite{Ati-2001} and \cite{Ati-Sut-2002}. Its origin lies in physics, in an article by Sir Michael Berry and Jonathan Robbins (\cite{BR1997}). In short, in \cite{BR1997}, the authors propose a geometric construction to explain the sign arising in the spin-statistics theorem as a kind of Berry phase. Their construction is mostly for the case of $n = 2$ identical particles (with work on the $n = 3$ case too). In order to extend their construction to an arbitrary number $n$ of identical particles, they faced a technical hurdle, which led ot the formulation of the Berry--Robbins problem in the literature. A beautiful way to solve the Berry--Robbins problem was done by Sir Michael Atiyah and Roger Bielawski in \cite{Ati-Bie-2002} using properties of Nahm's equations. An alternative, more elementary and more explicit construction is proposed by Atiyah in \cite{Ati-2000} and \cite{Ati-2001}, provided a certain linear independence conjecture holds. Stronger conjectures are formulated in \cite{Ati-Sut-2002} using Atiyah's normalized determinant function. This is the Atiyah problem on configurations of points. We briefly go over the problem in this section.

Let $n \geq 2$ be an integer. Denote by $C_n(\R^3)$ the (ordered) configuration space of $n$ distinct points in $\R^3$. Let
\[ \mathbf{x} := (\mathbf{x}_1, \dots, \mathbf{x}_n) \in C_n(\R^3), \]
i.e. $\mathbf{x}_i \in \R^3$ for $i = 1, \dots, n$ and the $\mathbf{x}_i$ are distinct.

For $i, j \in [n] := \{1, \dots, n\}$ distinct, let
\[ v_{ij} := \frac{\mathbf{x}_j - \mathbf{x}_i}{\lVert \mathbf{x}_j - \mathbf{x}_i \rVert} \in S^2. \]

We identify $S^2$ with the Riemann sphere using, for example, stereographic projection from the ``North'' pole, i.e. point $(0, 0, 1)^T$. Let $t_{ij}$ denote the point on the Riemann sphere corresponding to $v_{ij}$. We then define, for $i = 1, \dots, n$,
\[ p_i(t) = \prod_{j \neq i} (t - t_{ij}). \]

Sir Michael Atiyah conjectured that, no matter which $\mathbf{x} \in C_n(\R^3)$ one starts with, the $p_i$ ($1 \leq i \leq n$) are linearly independent over $\mathbb{C}$. This is known as Conjecture 1.

Let $A$ be the complex $n$ by $n$ matrix having the coefficients of $p_j$ as its $j$-th column. The determinant of $A$, after proper normalization, is known as the Atiyah determinant function, and denoted by $D$. The function $D$ is a smooth complex-valued function on $C_n(\R^3)$, which is invariant under the symmetric group $S_n$ (which permutes the points $\mathbf{x}_i$ comprising a configuration $\mathbf{x}$), is also invariant under proper Euclidean transformations in $\mathbb{R}^3$ and gets complex conjugated under improper ones.

A stronger conjecture was made by Atiyah and Sutcliffe in \cite{Ati-Sut-2002}, namely that
\[ |D(\mathbf{x})| \geq 1, \]
for any $\mathbf{x} \in C_n(\R^3)$. The latter is known as Conjecture 2. It is clear that Conjecture 2 implies Conjecture 1.

Atiyah and Sutcliffe also formulated another conjecture, known as Conjecture 3, which is stronger than Conjecture 2 and is quite interesting, but we shall refrain from discussing it here.

We note that Conjectures 1 and 2 have been proved if $n \leq 4$ (cf. \cite{KhuJoh2014}, \cite{EasNor2001}, \cite{Mal2023} and talks by D. Svrtan, which can be found online, on the $n = 4$ and planar $n = 5$ cases).

Our main motivation in this work is in generalizing the Atiyah problem on configurations using finite graphs. We hope this will shed some more light on the problem, and believe that the conjectured geometric inequalities we formulate in this work are interesting in their own right.

Our construction makes use of finite simple graphs, configurations of points in $\R^3$ and tensors. We will define a complex-valued function which we call $G$-amplitude on a configuration space of points associated to a finite simple graph $G$ and conjecture some interesting properties of this function. If the graph is the complete graph, our Conjecture A reduces to Conjecture 1, while Conjecture B (Conjectures A and B are formulated in section \ref{conjectures}) is related to Conjecture 2.

\subsection*{Acknowledgements}

The author is endebted to Dennis Sullivan, Peter Olver and Niky Kamran for reading and commenting on several of his ideas regarding the Atiyah problem on configurations, including this one. The author would like to thank the anonymous reviewer for their comments which resulted in a clearer manuscript. He would also like to thank Stephen Drury for informing him about Marcus's inequality and the related work by Elliott Lieb and his famous conjecture. He also acknowledges having had illuminating discussions with Felix Huber, Jonathan Robbins and John Baez on the problem from a more physical perspective.

\section{The amplitude function associated to a finite simple graph}

Let $G = (V, E)$ denote a non-oriented finite simple graph with set of vertices $V$ and set of edges $E$. We thus do not allow for loops (an edge joining a vertex to itself), nor do we allow for multiple edges between two distinct vertices. Note that the construction we will present for the amplitude function (to be defined later in this article), as well as the related conjectures, are sensible even if the graph has multiple edges, but we did not do many numerical studies with graphs having multiple edges.

Assume that
\[ V = \{v_1, \dots, v_n \}. \]
We thus can represent each edge $e \in E$ by an unordered pair of (distinct) vertices.

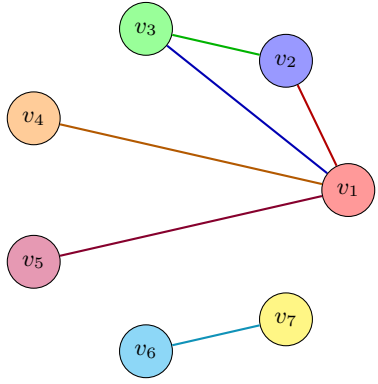
\begin{figure}
\begin{center}
\begin{tikzpicture}[scale=1]
  \def\n{7}      
  \def\r{2.2}    

  \readlist\vertexcolors{red!40, blue!40, green!40, orange!40, purple!40, cyan!40, yellow!60}
  \readlist\edgecolors{red!70!black, blue!70!black, green!70!black, orange!70!black, purple!70!black, cyan!70!black}

  \def\edges{1/2, 1/3, 2/3, 1/4, 1/5, 6/7}

 \foreach \i in {1,...,\n}{
    \pgfmathtruncatemacro{\cindex}{mod(\i-1,\vertexcolorslen)+1}
    \edef\thisfill{\vertexcolors[\cindex]}
    \node[vertex, fill=\thisfill] (v\i) at ({360/\n*(\i-1)}:\r) {$v_{\i}$};
  }

  \foreach [count=\ei] \u/\v in \edges{
    \pgfmathtruncatemacro{\ecindex}{mod(\ei-1,\edgecolorslen)+1}
    \edef\thisdraw{\edgecolors[\ecindex]}
    \draw[edge, draw=\thisdraw] (v\u) -- (v\v);
  }

\end{tikzpicture}
\caption{A finite simple graph $G$ with $7$ vertices and $6$ edges.}
\label{fig:graph1}
\end{center}
\end{figure}

Fix a non-oriented finite simple graph $G = (V, E)$, with $n = |V|$. We denote by $C_G(\R^3)$ the following configuration space
\[ C_G(\R^3) = \{ \mathbf{x}: V \to \R^3;\, \text{$\mathbf{x}(v_i) \neq \mathbf{x}(v_j)$ whenever $\{v_i, v_j\} \in E$} \}. \]
We will construct a smooth complex-valued function $\mathcal{A}_G$ on $C_G(\R^3)$, which we call the $G$-amplitude function. Let
\[ \mathbf{x}_i := \mathbf{x}(v_i) \]
and write
\[ \mathbf{x} = (\mathbf{x}_1, \dots, \mathbf{x}_n), \]
by a slight abuse of notation. We let
\begin{align}
\sigma_1 &= \begin{pmatrix} 0 & 1 \\ 1 & 0 \end{pmatrix} \\
\sigma_2 &= \begin{pmatrix} 0 & -i \\ i & 0 \end{pmatrix} \\
\sigma_3 &= \begin{pmatrix} 1 & 0 \\ 0 & -1 \end{pmatrix}
\end{align}
be the Pauli matrices. We also denote by $\vec{\sigma} = (\sigma_1, \sigma_2, \sigma_3)$ the $3$-vector of Pauli matrices.

Given $i, j \in [n] := \{1, \dots, n \}$ distinct, let
\begin{equation} v_{ij} = \frac{\mathbf{x}_j - \mathbf{x}_i}{\norm{\mathbf{x}_j - \mathbf{x}_i}} \in S^2. \end{equation}
Then $v_{ij}.\vec{\sigma}$ is a Hermitian $2$ by $2$ matrix with eigenvalues $\pm 1$. Denote by $\psi_{ij}$ a choice of normalized $1$-eigenvector of $v_{ij}.\vec{\sigma}$. In other words
\begin{equation} (v_{ij}.\vec{\sigma}) \psi_{ij} = \psi_{ij}. \end{equation}
Note that $\psi_{ij} \in S^3 \subset \mathbb{C}^2$ is defined up to a phase factor.

Let $S_G = \{ (i, j) ;\, \{v_i, v_j\} \in E \}$. Thus for example, if $G$ is as in fig. \ref{fig:graph1}, then
\[ S_G = \{ (1, 2), (1, 3), (1, 4), (1, 5), (2, 1), (2, 3), (3, 1), (3, 2), (4, 1), (5, 1), (6, 7), (7, 6) \}. \]
Let $P_G$ be the group of all permutations of $S_G$ which preserve the "source", i.e. the first index, of each element in $S_G$.

Let $Q_G$ be the group of all permutations of $S_G$ which map any element $(i, j) \in S_G$ to either $(i, j)$ or $(j, i)$. In other words, if one partitions $S_G$ as the union of sets of the form $\{ (i, j), (j, i) \}$, with each such set corresponding to an (unordered) edge of $G$, then $Q_G$ is the group of all permutations of $S_G$ which preserve this partition.

Let
\begin{equation} \widetilde{T}_G(\mathbf{x}) = \otimes_{(i, j) \in S_G} \psi_{ij} \in \bigotimes_{(i, j) \in S_G}\C^2. \end{equation}
Note that $\widetilde{T}_G(\mathbf{x})$ is defined up to a phase factor due to the phase factor ambiguity in each factor $\psi_{ij}$. It is possible to remove this phase ambiguity by proper normalization. Let $\omega \in \Lambda^2(\C^2)^*$ be defined by
\begin{equation} \omega\left(\begin{pmatrix} u_1 \\ v_1 \end{pmatrix}, \begin{pmatrix} u_2 \\ v_2 \end{pmatrix}\right) := u_1 v_2 - u_2 v_1. \label{Omega} \end{equation}
Choose a directed graph, say $\Gamma$, whose underlying undirected graph (i.e. by ``forgetting'' that each edge has a preferred direction) is $G$. Note that, using $\Gamma$, for each $\{(i, j), (j, i) \} \subseteq S_G$, precisely one of $(i, j)$, $(j, i)$ corresponds to the preferred direction of the corresponding edge. For example, if the edge goes from vertex $v_i$ to $v_j$ (resp. $v_j$ to $v_i$), then $(i, j)$ (resp. $(j, i)$) corresponds to that preferred direction. We identify $\Delta_\Gamma$ with the subset of $S_G$ corresponding to the chosen edge directions in $\Gamma$ (the choice of notation $\Delta$ is due to an analogy with a choice of positive roots in Lie theory). Then, if one defines
\begin{equation} T_\Gamma(\mathbf{x}) :=  \otimes_{(i, j) \in S_G} \psi_{ij} / \prod_{(i, j) \in \Delta_\Gamma} \omega(\psi_{ij}, \psi_{ji}), \label{def-T} \end{equation}
then $T_\Gamma(\mathbf{x})$ is well defined, including its phase. We also note that the denominator does not vanish, since $\psi_{ij}$ and $\psi_{ji}$ are nonzero and orthogonal, being the $\pm 1$ eigenvectors of $v_{ij}.\vec{\sigma}$, and therefore linearly independent.

A permutation $\sigma$ of $S_G$ acts on the right on an element in $\bigotimes_{(i, j) \in S_G}\C^2$, say 
\[ T = \otimes_{(i, j) \in S_G} \phi_{ij}, \] 
by permuting the corresponding vectors $\phi_{ij}$. More precisely,
\[ T.\sigma = \otimes_{(i, j) \in S_G} \phi_{\sigma(i, j)}. \]
Thus the permutation $\sigma$ maps the tensor $T$ to $T.\sigma$, where the vector in $T.\sigma$ sitting at the location indexed by $(i, j) \in S_G$ is the vector in $T$ sitting at the location indexed by $\sigma(i, j)$. This action extends by linearity on more general tensors, which are not necessarily decomposable.
If one has two permutations, say $\sigma_i$ ($i = 1, 2$) of $S_G$, then
\[ (T.\sigma_1).\sigma_2 = T.(\sigma_1 \sigma_2), \]
so this defines a well defined right action of the permutation group of $S_G$ on $\bigotimes_{(i, j) \in S_G}\C^2$.

We now define
\begin{align} 
p_G &= \sum_{\sigma \in P_G} \sigma \in \mathbb{Z}[P_G] 
\subset \mathbb{Z}[\operatorname{Aut}(S_G)],  \\
q_G &= \sum_{\tau \in Q_G} \sgn(\tau) \,\tau \in \mathbb{Z}[Q_G] \subset \mathbb{Z}[\operatorname{Aut}(S_G)] .
\end{align}
We finally associate to $\mathbf{x}$ the symmetrization of $T_\Gamma(\mathbf{x})$ by first applying $p_G$, then $q_G$, i.e.
\begin{equation} \mathbf{x} \mapsto T_\Gamma(\mathbf{x}).(p_G q_G). \end{equation}
But $T_\Gamma(\mathbf{x}).(p_G q_G)$ lives in a complex one-dimensional subspace, the image of $q_G$ (indeed, the image of $q_G$ is the tensor product of $|E|$ copies of $\Lambda^2 \mathbb{C}^2 \simeq \mathbb{C}$), so we have essentially defined a complex-valued function on $C_G(\R^3)$, once one chooses a nonzero element in the image of $q_G$. One natural choice for such an element, which depends on $\Gamma$, is

\[ (\otimes_{(i, j) \in S_G} \Psi_{ij}) . q_G,\]

where 
\[ \Psi_{ij} = \begin{cases} e_1 := (1, 0)^T, & \text{if $(i, j) \in \Delta_\Gamma$}, \\ 
e_2 := (0, 1)^T, & \text{if $(i, j) \notin \Delta_\Gamma$}. \end{cases} \]

One may then write

\begin{equation} T_\Gamma(\mathbf{x}).(p_G q_G) = \mathcal{A}_G(\mathbf{x}) (\otimes_{(i, j) \in S_G} \Psi_{ij}) . q_G, \label{def-amplitude} \end{equation}

for some complex-valued function $\mathcal{A}_G$ on $C_G(\mathbb{R}^3)$. It is to be noted that $\mathcal{A}_G$ depends on $G$, but does not depend on the choice of oriented graph $\Gamma$ corresponding to $G$. Indeed, if one chooses, say, $\Gamma'$ instead of $\Gamma$, corresponding to the same graph $G$, then both sides of the previous equation get multiplied by the same sign factor, which is $(-1)^s$, where $s$ is the number of oriented edges, i.e. elements of $S_G$, whose orientation gets reversed when going from $\Gamma$ to $\Gamma'$.

We may proceed in another equivalent way. We first form the tensor $T_\Gamma(\mathbf{x}).p_G$. Then, we enumerate the edges. For example, assume that
$$ E = \{ e_1, \dots, e_m \}. $$
Let $e \in E$, so that $e = \{v_i, v_j \}$, which is oriented in $\Gamma$ from $v_i$ to $v_j$, say. Denote by $\omega_e$ the process of taking a tensor having in particular indices labeled by $(i, j)$ and $(j, i)$ and then contracting the index labeled by $(i, j)$ with that labeled by $(j, i)$ (in that order) using the complex skew-symmetric bilinear form $\omega$. Since we are acting by permutations of $S_G$ on the right, it makes sense to also act by $\omega_e$ on the right, so that for example in the composition $p_G \circ \omega_{e_i}$, say, $p_G$ is applied first and then $\omega_{e_i}$. We then define
\begin{equation} \mathcal{A}_G(\mathbf{x}) = (T_\Gamma(\mathbf{x}).p_G) \omega_{e_1} \circ \dots \circ \omega_{e_m}. \label{def-amp} \end{equation}
Note that the right-hand side does not depend on a choice of $\Gamma$, since a different choice would lead to a sign factor in the normalization factor of $T_\Gamma(\mathbf{x})$ in Eq. \eqref{def-T} and the same sign factor arising from applying the $\omega_e$ in Eq. \eqref{def-amp}, so that the two sign factors cancel out. Moreover, the right-hand side does not depend on how we enumerate the edges, since the $\omega_{e_i}$ commute, because they act on different pairs of indices. We have thus associated to each non-oriented finite simple graph $G$ a complex-valued function $\mathcal{A}_G$ on $C_G(\R^3)$, which we call the $G$-amplitude function. Note also that this definition of $\mathcal{A}_G$ agrees with our previous one in \eqref{def-amplitude}, in a way similar to how, for example, if $\phi_1, \phi_2 \in \mathbb{C}^2$, then 
\[ \phi_1 \wedge \phi_2 = \omega(\phi_1, \phi_2) e_1 \wedge e_2, \]
where $e_1, e_2$ is the canonical basis of $\mathbb{C}^2$.

\emph{For the remainder of this section, we assume WLOG that $V = [n] := \{1, \dots, n\}$.}

We now list some properties of $\mathcal{A}_G$. If $\sigma$ is any permutation of $V$, we define
$$ \mathbf{x}.\sigma = (\mathbf{x}_{\sigma(1)}, \dots, \mathbf{x}_{\sigma(n)}).$$
This defines a right action of the permutation group $\operatorname{Aut}(V)$ of $V$ on $C_G(\R^3)$. 

Given a permutation $\sigma$ of $V$, we define $\sigma.G$ to be the graph with set of vertices $V$, but with set of edges
$$ \sigma.E = \{ \{ \sigma(i), \sigma(j) \} ; \, \{i, j\} \in E \}. $$
Note that this defines a left action of $\operatorname{Aut}(V)$ on the set of finite simple graphs with set of vertices equal to $V$ (whence the notation, $\sigma.G$, and not $G.\sigma$, say).

Similarly, if $\Gamma = (V, \mathcal{E})$ is an oriented graph, we define $\sigma.\Gamma$ to be the oriented graph with set of vertices $V$, but with set of edges
$$ \sigma.\mathcal{E} = \{ (\sigma(i), \sigma(j)) ; \, (i, j) \in \mathcal{E} \}. $$

A permutation $\sigma \in \operatorname{Aut}(V)$ is said to be a \emph{symmetry} of $G$ if $\sigma.G = G$ or, equivalently, if
$$ \forall i, j \in [n], i \text{ is a $G$-neighbor of } j \iff \sigma(i) \text{ is a $G$-neighbor of } \sigma(j), $$
where two vertices in $G$ are $G$-neighbors iff there is an edge in $G$ joining them. We denote the group of symmetries of $G$ by $\operatorname{Aut}(G)$.

\begin{definition} Given a symmetry $\sigma \in \operatorname{Aut}(G)$ of $G$, we define $s_G(\sigma)$ to be $1$, resp. $-1$, if the number of edge orientation flips by going from $\Gamma$ to $\sigma.\Gamma$ is even, resp. odd.
\end{definition}
We note that $s_G(\sigma)$ depends only on $G$ and $\sigma$ and does not depend on the choice of orientation $\Gamma$ of $G$.

\begin{proposition} For any permutation $\sigma \in \operatorname{Aut}(V)$ of $V$, we have
\begin{equation} T_\Gamma(\mathbf{x}.\sigma) =
T_{\sigma.\Gamma}(\mathbf{x}).\sigma. \label{eq1-prop1}\end{equation}
If, in addition, $\sigma$ is a symmetry of $G$, we then also have that
\begin{equation} T_\Gamma(\mathbf{x}.\sigma) = s_G(\sigma) T_\Gamma(\mathbf{x}).\sigma.
\label{eq2-prop1}
\end{equation}

\label{prop1}
\end{proposition}

\begin{proof}
Note first that the tensor on the right-hand side of Eq. \eqref{eq1-prop1} is labeled by
\[ \sigma^{-1} S_{\sigma.G} = S_G,\]
so that the tensors on both sides of Eq. \eqref{eq1-prop1} are labeled by the same set, $S_G$. 
If we write $i' = \sigma(i)$, and $\mathbf{x}' := \mathbf{x}.\sigma$, we then have
\[ \psi_{ij}(\mathbf{x}') = \psi_{i'j'}(\mathbf{x}), \]
up to a phase factor, which we can WLOG assume to be equal to $1$. As a result, we also have that
\[ \omega(\psi_{ij}(\mathbf{x}'), \psi_{ji}(\mathbf{x}')) = \omega(\psi_{i'j'}(\mathbf{x}), \psi_{j'i'}(\mathbf{x})).\]
Eq. \eqref{eq1-prop1} follows from combining these two facts. Finally, Eq. \eqref{eq2-prop1} follows from Eq. \eqref{eq1-prop1} by noting that
\[ T_{\sigma.\Gamma}(\mathbf{x}).\sigma = s_G(\sigma) T_\Gamma(\mathbf{x}).\sigma,\]
assuming $\sigma$ is a symmetry of $G$.
\end{proof}

We note the following corollary.

\begin{corollary}  The $G$-amplitude function $\mathcal{A}_G$ is equivariant under the group $\operatorname{Aut}(V)$ of permutations of the set $V$ of vertices of the graph $G$. More precisely,
\begin{equation} \mathcal{A}_G(\mathbf{x}.\sigma) = \mathcal{A}_{\sigma.G}(\mathbf{x}), \label{equi1} \end{equation}
for all $\mathbf{x} \in C_G(\R^3)$ and all $\sigma \in \operatorname{Aut}(V)$.

In particular, we note that $\mathcal{A}_G$ is invariant under the group $\operatorname{Aut}(G)$ of symmetries of $G$, i.e. that
\begin{equation} \mathcal{A}_G(\mathbf{x}.\sigma) = \mathcal{A}_G(\mathbf{x}), \label{equi2} \end{equation}
for all $\mathbf{x} \in C_G(\R^3)$ and all $\sigma \in \operatorname{Aut}(G)$.
\end{corollary}

\begin{proof}
Let $\Omega_\Gamma$ be the composition of the edge contractions with respect to the oriented graph $\Gamma$ in some order (the order does not matter since edge contractions corresponding to different edges commute). We apply $p_G.\Omega_\Gamma$ from the right on both sides of Eq. \eqref{eq1-prop1} and obtain
\[ T_\Gamma(\mathbf{x}.\sigma).p_G.\Omega_\Gamma = T_{\sigma.\Gamma}(\mathbf{x}).\sigma.p_G.\Omega_\Gamma
= T_{\sigma.\Gamma}(\mathbf{x}).p_{\sigma.G}.\sigma.\Omega_\Gamma, \]
using the formula
\[ \sigma.p_G = p_{\sigma.G}.\sigma. \]
Indeed, if a permutation $\tau \in \operatorname{Aut}(S_G)$ occurs in $p_G$, then $\sigma \tau \sigma^{-1} \in \operatorname{Aut}(S_{\sigma.G})$ occurs in $p_{\sigma.G}$. Moreover, $\sigma.\Omega_\Gamma$ can be replaced with $\Omega_{\sigma.\Gamma}$ (remember that we are first applying $\sigma$, then $\Omega_\Gamma$), from which we deduce that
\[ T_\Gamma(\mathbf{x}.\sigma).p_G.\Omega_\Gamma = T_{\sigma.\Gamma}(\mathbf{x}).p_{\sigma.G}.\Omega_{\sigma.\Gamma}. \]
This finishes the proof of Eq. \eqref{equi1}. Moreover, Eq. \eqref{equi2} now follows by letting $\sigma$ be an element of the subgroup $\operatorname{Aut}(G)$ of $\operatorname{Aut}(V)$.
\end{proof}

Let $g \in SU(2)$. Note that $SU(2)$ acts on $\R^3$ via its adjoint representation, which thus induces an action of $SU(2)$ on $C_G(\R^3)$. We will denote this action simply by
$$ (g, x) \mapsto g.x. $$

Moreover, $SU(2)$ acts naturally on $\C^2$, so that $g \in SU(2)$ acts on each $\psi_{ij}$ ($(i, j) \in S_G$) simultaneously and thus also on $\bigotimes_{(i, j) \in S_G} \C^2$. We will denote this action simply by
$$ (g, T) \mapsto g. T. $$

We then have
\begin{proposition}For any $g \in SU(2)$, we have
\begin{equation} T_\Gamma(g.\mathbf{x}) = g.T_\Gamma(\mathbf{x}), \end{equation}
for any $\mathbf{x} \in C_G(\R^3)$, i.e. $T_\Gamma$ is equivariant under the action of $SU(2)$. We also have
\begin{equation} \mathcal{A}_G(g.\mathbf{x}) = \mathcal{A}_G(\mathbf{x}), \end{equation}
for any $\mathbf{x} \in C_G(\R^3)$. In other words, $\mathcal{A}_G$ is invariant under $SU(2)$, but $SU(2)$ acts on $C_G(\R^3)$ via $SO(3)$, so $\mathcal{A}_G$ is invariant under $SO(3)$.
\end{proposition}
\begin{proof} The complex skew-symmetric bilinear form $\omega$ is invariant under the action of $SU(2)$, from which the proposition follows. \end{proof}

Define $\tilde{j}$ to be the complex anti-linear map from $\C^2$ to itself mapping
\[ \tilde{j}(u, v)^T = (-\bar{v}, \bar{u})^T, \]
where $T$ denotes the transpose. Note that $\tilde{j}^2 = - \mathbf{1}$. We also denote by $\tilde{j}$ the induced action of the previous map on $\bigotimes_{(i, j) \in S_G} \C^2$.

In the mathematics literature, $\tilde{j}$ is known as the standard quaternionic structure on $\mathbb{C}^2$, and in the physics literature, it corresponds to the time-reversal operator, acting on spin-$1/2$ states.

We then have
\begin{proposition}
\[ T_\Gamma(-\mathbf{x}) = \tilde{j} T_\Gamma(\mathbf{x}), \]
for any $\mathbf{x} \in C_G(\R^3)$, from which we deduce that
\begin{equation} \mathcal{A}_G(-\mathbf{x}) = \overline{\mathcal{A}_G(\mathbf{x})}, \end{equation} 
for any $\mathbf{x} \in C_G(\R^3)$.
\end{proposition}

\begin{corollary}  The $G$-amplitude function $\mathcal{A}_G$ gets conjugated under improper orthogonal transformations of $\mathbb{R}^3$. \end{corollary}

It is also clear that both $T_\Gamma$ and $\mathcal{A}_G$ are invariant under translations in $\R^3$. We have thus proved
\begin{proposition} The $G$-amplitude function $\mathcal{A}_G$ is invariant under proper Euclidean transformations of $\R^3$ and gets complex conjugated under improper ones. \end{proposition}

We also note that $\mathcal{A}_G$ possesses a multiplicative property with respect to $G$.
\begin{proposition} If $G = G_1 \sqcup G_2$ is the disjoint union of two subgraphs, $G_1$ and $G_2$ (with no vertices in common, and thus no edges in common either), then
\[ \mathcal{A}_G(\mathbf{x}) = \mathcal{A}_{G_1}(\mathbf{x}) \mathcal{A}_{G_2}(\mathbf{x}), \]
for any $\mathbf{x} \in C_G(\R^3)$. \label{multiplicative} \end{proposition}

\section{Conjectures} \label{conjectures}

In this section, we assume that $G$ is a non-oriented finite simple graph.

\begin{conjecture} The function $\mathcal{A}_G$ is nowhere vanishing on $C_G(\R^3)$. \end{conjecture}

\begin{definition} Given a non-oriented finite simple graph $G$, we define
\begin{equation} c_G = \inf \{ |\mathcal{A}_G(\mathbf{x})| ;\, \mathbf{x} \in C_G(\R^3) \} \end{equation}
\end{definition}

We conjecture

\begin{conjecture} $c_G \geq 1$.
\end{conjecture}

We remark that Conjecture B implies Conjecture A. Note that for a given graph $G$, it may well be (at least based on numerical simulations) that $c_G > 1$. However, the lower bound in Conjecture B seems to be optimal (based on numerical simulations) when applied to \emph{all} finite simple graphs $G$. It would be interesting to try to express $c_G$ in terms of properties of $G$.

In the case where $G = T$ is a tree, we conjecture the following.

\begin{conjecture} If $T$ is a tree (a connected circuit-free non-oriented finite simple graph), then
\[ \Real(\mathcal{A}_T(\mathbf{x})) \geq 1, \]
for any $\mathbf{x} \in C_T(\R^3)$. \label{main-theorem}
\end{conjecture}
If true, Conjecture C implies that Conjecture B (and hence also conjecture A)  holds for trees.

We remark that Conjecture C is not in general sharp for any given tree $T$, but it seems to be sharp (at least based on numerical simulations) when applied to \emph{all} trees.

\begin{figure}
\begin{center}
\begin{tikzpicture}[scale=1]
  \def\n{5}      
  \def\r{2.2}    

  \readlist\vertexcolors{red!40, blue!40, green!40, orange!40, purple!40, cyan!40, yellow!60}
  \readlist\edgecolors{red!70!black, blue!70!black, green!70!black, orange!70!black, purple!70!black, cyan!70!black}

  \def\edges{1/2, 1/3, 1/4, 1/5, 2/3, 2/4, 2/5, 3/4, 3/5, 4/5}

 \foreach \i in {1,...,\n}{
    \pgfmathtruncatemacro{\cindex}{mod(\i-1,\vertexcolorslen)+1}
    \edef\thisfill{\vertexcolors[\cindex]}
    \node[vertex, fill=\thisfill] (v\i) at ({360/\n*(\i-1)}:\r) {$v_{\i}$};
  }

  \foreach [count=\ei] \u/\v in \edges{
    \pgfmathtruncatemacro{\ecindex}{mod(\ei-1,\edgecolorslen)+1}
    \edef\thisdraw{\edgecolors[\ecindex]}
    \draw[edge, draw=\thisdraw] (v\u) -- (v\v);
  }

\end{tikzpicture}
\caption{The complete graph $K_5$ with $5$ vertices.}
\label{fig:complete-graph}
\end{center}
\end{figure}
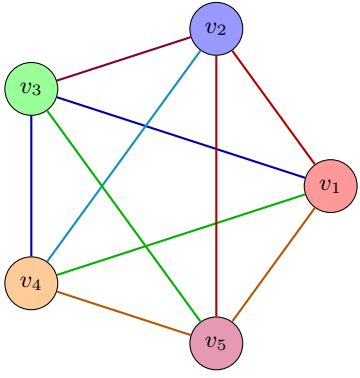

We now note the link with the Atiyah problem on configurations of points. If $G = K_n$ is the complete graph with $n$ vertices (see Fig. \ref{fig:complete-graph}), it can be shown that $\mathcal{A}_{K_n}$ has the same symmetry properties as the normalized Atiyah determinant function $D_n$, so that
\[ \mathcal{A}_{K_n} = a_n D_n, \]
where $a_n$ is a constant depending on $n$ alone. If $\mathbf{x} \in C_{K_n}(\R^3)$ is a collinear configuration (all the points in the configuration lie on a single line in $\R^3$), then it is known that $D$ takes the value $1$ in this case. Working out what $\mathcal{A}_{K_n}$ is equal to, in this case, gives
\[ a_n = \prod_{k = 1}^{n - 1} (k!)^2. \]

Hence the \AS Conjecture 2 is equivalent in our language to the conjecture that
\[ c_{K_n} = a_n, \]
which, at the time of writing, is only known in general for $n \leq 4$.

Indeed,  this was  the main motivation for our construction: to generalize the Atiyah problem on configurations of points, so that it hopefully becomes in its proper setting, thus paving the way for an eventual solution.

\section{A conjecture in matrix analysis}

In this section, we state a conjecture in matrix analysis, which implies Conjecture C. Given an equivalence relation, say $\sim$, on $[n]$, we let $P_\sim$ be the group of all permutations of $[n]$ which preserve the equivalence classes of $[n]$ under $\sim$. Let $\sim_i$, for $i = 1, 2$, be two equivalence relations on $[n] = \{1, \dots, n\}$, with associated permutation groups $P_1$, $P_2$, respectively (we write $P_i$ rather than $P_{\sim_i}$ to alleviate the notation).

Let $H$ be a Hermitian positive semidefinite $n \times n$ matrix. We define
\[ f_{\sim_1, \sim_2}(H) = \sum_{\sigma_1 \in P_1} \sum_{\sigma_2 \in P_2} \prod_{i = 1}^n H_{\sigma_1(i), \sigma_2(i)}. \]

\begin{conjecture} Under the above hypotheses, we have
\[ \Real( f_{\sim_1, \sim_2}(H) ) \geq |P_1 \cap P_2| \, \prod_{i = 1}^n H_{i i}. \]
\end{conjecture}

We remark that by taking one of the equivalence relation to be the discrete one and the other one to be the indiscrete one, one recovers Marcus's permanent analogue of the Hadamard inequality (\cite{Mar1963}).

\begin{theorem} Conjecture D implies Conjecture C. \end{theorem}

\begin{figure}
\begin{center}
\begin{tikzpicture}[
  level distance=11mm,
  level 1/.style={sibling distance=36mm,
                  every node/.style={circle, draw=blue!70!black, fill=blue!20},
                  edge from parent/.style={thick, draw=black, -Latex}},
  level 2/.style={sibling distance=22mm,
                  every node/.style={circle, draw=green!70!black, fill=green!20},
                  edge from parent/.style={thick, draw=green!70!black, Latex-}}, 
  level 3/.style={sibling distance=12mm,
                  every node/.style={circle, draw=orange!70!black, fill=orange!25},
                  edge from parent/.style={thick, draw=green!70!black, -Latex}}, 
  every node/.style={circle, draw, inner sep=1.2pt, minimum size=6mm},
  edge from parent/.style={draw=blue!60!black, thick, -Latex}, 
  edge from parent path={(\tikzparentnode.south) -- (\tikzchildnode.north)}
]
\node {\vlabel}                
  child { node {\vlabel}       
    child { node {\vlabel}     
      child { node {\vlabel} } 
    }
    child { node {\vlabel} }   
  }
  child { node {\vlabel}       
    child { node {\vlabel}     
      child { node {\vlabel} } 
    }
  };
\end{tikzpicture}
\caption{A finite tree with $8$ vertices and $4$ levels}
\label{fig:tree}
\end{center}
\end{figure}
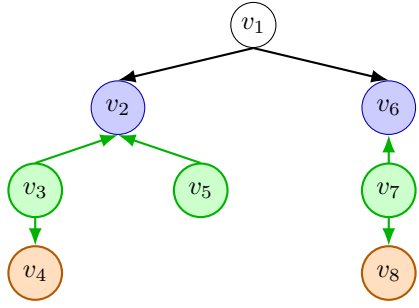

\begin{proof} 

Let $T$ be a finite tree. Choose one of the vertices as root. Let level $0$ denote the root level, then let level $1$ denote the next level and so on. If $l$ is even (resp. odd), orient any edge joining a vertex from level $l$ with a vertex from level $l+1$ from the parent vertex to the child vertex (resp. from the child vertex to the parent vertex), as in Fig. \ref{fig:tree}.

Label the edges as $e_1, \dots, e_m$, oriented as discussed in the previous paragraph. Note that, if $\psi$ and $\phi$ are vectors in $\C^2$, then
\begin{equation} \omega(\psi, \tilde{j} \phi) = \langle \phi,\, \psi \rangle. \label{identity} \end{equation}
Given $\mathbf{x} \in C_T(\R^3)$, let
\[ H = ( \langle \psi_{e_a}, \psi_{e_b} \rangle ),\]
($1 \leq a, b \leq m$) be the Gram matrix of the corresponding $\psi_{e_a}$, for $a = 1, \dots, m$ (with respect to the Hermitian inner product $\langle -, -\rangle$ on $\C^2$). Let $e_a = (v_{i_a}, v_{j_a})$ and $e_b = (v_{i_b}, v_{j_b})$ be two edges (oriented as previously discussed) of $T$. We write that
\[ e_a \sim_1 e_b \iff v_{i_a} = v_{i_b} \]
and
\[ e_a \sim_2 e_b \iff v_{j_a} = v_{j_b}. \]
Both $\sim_1$ and $\sim_2$ are equivalence relations on the set of (properly oriented) edges of $T$. Using Eq. \eqref{identity}, one may see that
\[ \mathcal{A}_T(\mathbf{x}) = \overline{f_{\sim_1, \sim_2}(H)}. \]
Hence, we deduce that Conjecture D implies Conjecture C.
\end{proof}
It is easy to prove the following proposition, using a reindexing trick.
\begin{proposition} If $H$ is Hermitian positive semidefinite and one of the two equivalence relations $\sim_i$ ($i = 1, 2$) is finer than the other, then Conjecture D holds. \end{proposition}
\begin{proof} Assume WLOG that $\sim_1$ is finer than $\sim_2$, which corresponds to $P_1$ being a subgroup of $P_2$. We then have that
\[ f_{\sim_1, \sim_2}(H) = \sum_{\sigma_1 \in P_1} \sum_{\sigma_2 \in P_2} \prod_{i = 1}^n H_{\sigma_1(i), \sigma_2(i)} =
\sum_{\sigma_1 \in P_1} \sum_{\sigma_2 \in P_2} \prod_{i = 1}^n H_{i, \sigma_2 \sigma_1^{-1}(i)}. \]
But, for a fixed $\sigma_1 \in P_1$, we have that
\[ \sum_{\sigma_2 \in P_2} \sigma_2 \sigma_1^{-1} = \sum_{\sigma_2 \in P_2} \sigma_2,\]
as elements of $\mathbb{Z}[P_2]$, from which we deduce that
\[ f_{\sim_1, \sim_2}(H) = |P_1| \sum_{\sigma_2 \in P_2} H_{i, \sigma_2(i)}\]
(note that $|P_1| = |P_1 \cap P_2|$). Since the right-hand side is the product of the permanents of the set of principal submatrices of $H$ corresponding to the partition of $[m] := \{1, \dots, m\}$ induced by $\sim_2$, applying Marcus's permanent inequality to each permanent finishes the proof of the proposition.
\end{proof}

A star graph is a graph with a root vertex, say $v_1$, and every other vertex, say $v_2, \dots, v_n$ connected to the root $v_1$, and no other edges (see fig. \ref{fig:star_graph}). If $\sim_1$ is the discrete equivalence relation and $\sim_2$ is the indiscrete one, then Conjecture D holds in this case, as it amounts to Marcus's permanent inequality (\cite{Mar1963}). We have thus proved the following.

\begin{proposition} Conjecture C is true if $T$ is a ``star'' graph.
\end{proposition}

\begin{figure}
\begin{tikzpicture}[
  vertex/.style={circle,draw=blue!70!black,fill=blue!15,inner sep=2pt,minimum size=20pt,font=\small},
  edge/.style={very thick,draw=purple!70!black}
]
  \node[vertex, draw=orange!80!black, fill=orange!25, minimum size=24pt] (v1) at (0,0) {$v_1$};

  \node[vertex, draw=green!60!black, fill=green!20] (v2) at (2.6, 1.2) {$v_2$};
  \node[vertex, draw=teal!60!black,  fill=teal!20]  (v3) at (2.6, 0.0) {$v_3$};

  \node[font=\large, text=gray!70] at (2.6, -1.1) {$\vdots$};

  \node[vertex, draw=red!70!black, fill=red!20] (vn) at (2.6, -2.2) {$v_n$};

  \draw[edge] (v1) -- (v2);
  \draw[edge] (v1) -- (v3);
  \draw[edge] (v1) -- (vn);
\end{tikzpicture}
\caption{A star graph}
\label{fig:star_graph}
\end{figure}

The linear graph $LG_n$ with $n$ vertices (cf. fig. \ref{fig:linear_graph}) is the tree having $n$ vertices, say $v_1, \dots, v_n$, with an edge connecting each $v_i$ with $v_{i+1}$, for $i = 1, \dots, n - 1$. 

\begin{theorem} Conjecture C holds for $LG_n$, for $2 \leq n \leq 5$. \label{main-thm} \end{theorem}

We prove Thm. \ref{main-thm} in the next section.

\begin{figure}
\begin{tikzpicture}[
  vertex/.style={circle,draw=blue!70!black,fill=blue!15,inner sep=2pt,minimum size=20pt,font=\small},
  edge/.style={very thick,draw=purple!70!black}
]

\node[vertex,draw=orange!80!black,fill=orange!25] (v1) at (0,0) {$v_1$};

\node[vertex,draw=green!60!black,fill=green!20] (v2) at (2,0) {$v_2$};

\node[vertex,draw=teal!70!black,fill=teal!20] (v3) at (4,0) {$v_3$};

\node[font=\large] (dots) at (5.5,0) {$\cdots$};

\node[vertex,draw=red!70!black,fill=red!20] (vn) at (7,0) {$v_n$};

\draw[edge] (v1) -- (v2);
\draw[edge] (v2) -- (v3);
\draw[edge] (v3) -- (dots);
\draw[edge] (dots) -- (vn);

\end{tikzpicture}
\caption{A ``linear'' graph}
\label{fig:linear_graph}
\end{figure}

\section{Proof of Theorem \ref{main-thm}}

In this proof, we let $\phi_1 = \psi_{12}$, $\phi_2 = \psi_{32}$, $\phi_3 = \psi_{34}$ and $\phi_4 = \psi_{54}$ and assume that the $\phi_i$ are normalized for $i = 1, \dots, n$, with respect to the standard Hermitian inner product $\langle -, - \rangle$ on $\mathbb{C}^2$.

Let $h_{ab} = \langle \phi_a, \phi_b \rangle$, for $1 \leq a, b \leq n$, and let $H = (h_{ab})$ be the Gram matrix of the $\phi_a$. We will write the Kronecker product of, say, $\phi_a$, $\phi_b$ and $\phi_c$, simply as $\phi_a \phi_b \phi_c$. 

\subsection*{Case $n = 2$} The $n = 2$ case is trivial, since the $LG_2$-amplitude is identically equal to $1$. 

\subsection*{Case $n = 3$} We have that
\[
\mathcal{A}_{LG_3}(\mathbf{x}) = 1 + \langle \phi_1, \phi_2 \rangle \langle \phi_2, \phi_1 \rangle = 1 + |\langle \phi_1, \phi_2 \rangle|^2,
\]
We thus see that Conjecture C holds for $LG_3$.

\subsection*{Case $n = 4$}
We have that
\begin{align*} & \operatorname{Re}(\mathcal{A}_{LG_4}(\mathbf{x})) \\
= & \, \lVert \phi_1 \phi_2 \phi_3 \rVert^2 + \langle \phi_2 \phi_1 \phi_3, \phi_1 \phi_2 \phi_3 \rangle + \langle \phi_1 \phi_2 \phi_3, \phi_1 \phi_3 \phi_2 \rangle + \operatorname{Re} \langle \phi_2 \phi_1 \phi_3, \phi_1 \phi_3 \phi_2 \rangle \\
= & \, 1 + |h_{12}|^2 + |h_{23}|^2 + \frac{1}{2} \left( h_{12} h_{23} h_{31} + h_{13} h_{32} h_{21} \right) \\
= & \, \frac{1}{2} + \frac{1}{2}|h_{12}|^2 + \frac{1}{2}|h_{23}|^2 + \frac{1}{2} \operatorname{per}(H) - \frac{1}{2}|h_{13}|^2.
\end{align*}

Lieb's inequality for Hermitian positive semidefinite block matrices (\cite{Lieb1966}) says the following. If
\[ H = \begin{pmatrix} A & B \\ B^* & C \end{pmatrix} \]
is a Hermitian positive semidefinite block matrix, where $A$ is $m \times m$, $B$ is $m \times n$ and $C$ is $n \times n$, then
\[ \operatorname{per}(H) \geq \operatorname{per}(A) \operatorname{per}(C).\]
Applying Lieb's inequality using the blocks corresponding to $\{1, 3\}$ and $\{2 \}$ of $H$ gives us that
\[ \operatorname{per}(H) \geq 1 + |h_{13}|^2. \]

Using this lower bound for the permanent of $H$ then gives the desired inequality, namely that
\[ \operatorname{Re}(\mathcal{A}_{LG_4}(\mathbf{x})) \geq 1, \]
for all $\mathbf{x} \in C_{LG_4}(\mathbb{R}^3)$.

\subsection*{Case $n = 5$} 

Let $a = h_{12}$, $b = h_{23}$ and $c = h_{34}$. Also write $u = h_{13}$ and $v = h_{24}$. By multiplying each $\phi_a$ by a phase factor $e^{i \\theta_a}$ (if needed), for $a = 1, \dots, 4$, we may assume, WLOG, that $a$, $b$ and $c$ are real and nonnegative. One may check that
\begin{equation} \operatorname{Re}(\mathcal{A}_{LG_4}(\mathbf{x})) = 1 + a^2 + b^2 + c^2 + a^2 c^2 + \operatorname{Re}(a b u + b c \bar{v} + a c u \bar{v}). \label{eq1} \end{equation}  We need the following lemma.

\begin{lemma} If
\[ \begin{pmatrix} 1 & \alpha & \gamma \\
\bar{\alpha} & 1 & \beta \\
\bar{\gamma} & \bar{\beta} & 1 \end{pmatrix}
\]
is Hermitian positive semidefinite, then
\[ \gamma = \alpha \beta + \sqrt{(1 - |\alpha|^2) (1 - |\beta|^2)} \xi, \]
for some $|\xi| \leq 1$.
\end{lemma}

\begin{proof}
The determinant of the matrix in the lemma is nonnegative, since the matrix is Hermitian positive semidefinite. Hence
\[ 1 + 2 \operatorname{Re}(\alpha \beta \bar{\gamma}) - |\alpha|^2 - |\beta|^2 - |\gamma|^2 \geq 0. \]
Completing the squares, this is equivalent to
\[ |\gamma - \alpha \beta|^2 \leq (1 - |\alpha|^2) (1 - |\beta|^2), \]
from which the proof follows.
\end{proof}

We apply this lemma to the principal minors $\{1, 2, 3\}$ and $\{2, 3, 4\}$ of $H$. We obtain
\[ u = ab + s \xi, \qquad v = bc + t \eta, \]
where
\[ s = \sqrt{(1 - a^2) (1 - b^2)}, \qquad t = \sqrt{(1 - b^2)(1 - c^2)} \]
and
\[ |\xi| \leq 1, \qquad |\eta| \leq 1. \]

Substituting into Eq. \eqref{eq1}, we get
\[ \operatorname{Re}(\mathcal{A}_{LG_5}(\mathbf{x}) )- 1 = P + \operatorname{Re}(A\xi + B \bar{\eta} + C \xi \bar{\eta}), \]
where
\[ P = a^2 + b^2 + c^2 + a^2 b^2 + a^2 c^2 + b^2 c^2 + a^2 b^2 c^2, \]
and
\[ A = abs(1 + c^2), \qquad B = bct(1 + a^2), \qquad C = acst. \]

Note that $A$, $B$, $C$ and $P$ are nonnegative, under our assumptions. Let
\[ \vec{x} = (\operatorname{Re}(\xi), \operatorname{Im}(\xi)), \qquad \vec{y} = (\operatorname{Re}(\eta), \operatorname{Im}(\eta)). \]
The problem is thus to minimize
\begin{equation} \Phi(\vec{x}, \vec{y}) = P + A e_1.\vec{x} + B e_1.\vec{y} + C \vec{x}.\vec{y}, \qquad \Vert \vec{x} \rVert \leq 1,  \lVert \vec{y} \rVert \leq 1. \end{equation}
We want to show that $\Phi(\vec{x},  \vec{y}) \geq 0$ for any $\vec{x}, \vec{y} \in \mathbb{R}^2$ such that  $\Vert \vec{x} \rVert \leq 1$ and $\lVert \vec{y} \rVert \leq 1$.

Fix $\vec{x} \in \mathbb{R}^2$, with $\lVert \vec{x} \rVert \leq 1$. With $\vec{x}$ fixed, $\Phi(\vec{x}, \vec{y})$ is an affine function of $\vec{y}$, so its minimum occurs at some $\vec{y}_0$, with $\lVert \vec{y}_0 \rVert  = 1$. Hence its minimum occurs at
\[ \vec{y}_0 = - \frac{B e_1 + C \vec{x}}{\lVert B e_1 + C \vec{x} \rVert}, \]
provided $B e_1 + C \vec{x} \neq 0$, and we have
\begin{equation} \Phi(\vec{x}, \vec{y}_0) = P + Ae_1.\vec{x} - \lVert B e_1 + C \vec{x} \rVert. \label{imp-eq} \end{equation}

The case where $B e_1 + C \vec{x} = 0$, leads to
\[ \Phi(\vec{x}, \vec{y}) = P + A e_1.\vec{x} \geq P - A. \]

We need another lemma.

\begin{lemma} The following upper bounds for $A$, $B$ and $C$ hold.
\begin{align*} 2 A & \leq (1 + c^2) (a^2 + b^2 - 2 a^2 b^2) \\
2 B & \leq (1 + a^2)(b^2 + c^2 - 2 b^2 c^2) \\
2 C & \leq (1 - b^2)(a^2 + c^2 - 2 a^2 c^2).
\end{align*} \label{lem2}
\end{lemma}

\begin{proof} Note that
\[ \frac{2A}{1+c^2} = 2(a\sqrt{1 - b^2})(b\sqrt{1 - a^2}) \leq a^2(1-b^2) + b^2(1-a^2), \]
proving the first inequality in the lemma. The other two inequalities can be proved in a similar way.
\end{proof}

Using the upper bound for $A$ in this lemma, we easily deduce that $P - A \geq 0$ (after expanding and simplifying). This shows that $\Phi(\vec{x}, \vec{y}) \geq 0$ in the case where $B e_1 + C \vec{x} = 0$.

We now let $t = \operatorname{Re}(\xi)$. We note that $|t| \leq 1$, since $\lVert \vec{x} \rVert \leq 1$. Using Eq. \eqref{imp-eq}, we then have that
\[ \phi(\vec{x}, \vec{y}_0) = P + At - \sqrt{B^2 + 2BCt + C^2 \lVert \vec{x} \rVert^2} \geq P + At - \sqrt{B^2 + 2BCt + C^2}. \]
Let
\[ g(t) = P + At - \sqrt{B^2 + 2BCt + C^2}, \]
for $t \in [-1, 1]$. If we can show that $g(t) \geq 0$ for any $t \in [-1, 1]$, this would be enough to imply that $\Phi(\vec{x}, \vec{y}) \geq 0$, for any $\vec{x}, \vec{y} \in \mathbb{R}^2$, such that $\lVert \vec{x} \rVert \leq 1$ and $\lVert \vec{y} \rVert \leq 1$.

WLOG, we may assume that $A$, $B$ and $C$ are all positive. Indeed, $g$ depends continuously on $A$, $B$ and $C$, so if we show that $g(t)$ is nonnegative for all $t \in [-1, 1]$ under the positivity assumption on $A$, $B$ and $C$, then a continuity argument would show that $g(t)$
would remain nonnegative if some (or all) of $A$, $B$ and $C$ vanish. We note that the expression inside the square root vanishes precisely at
\[
t_0 = -\frac{B^2 + C^2}{2 BC},
\]
so that $|t_0| \geq 1$. Hence $g$ is differentiable on $(-1, 1)$ and continuous on $[-1, 1]$. Hence $g$ attains its minimum either at a boundary point, or at some critical point which lies in $(-1, 1)$.

We first check the endpoint $t = -1$. We have
\[ g(-1) = P - A - |B - C|. \]
If $B \geq C$, then
\[ g(-1) = P - A - B + C \geq P - A - B. \]
Moreover, using lemma \ref{lem2}, we have
\begin{align*} & P - A - B \\
\begin{split}
\geq & \,a^2 + b^2 + c^2 + a^2b^2 + a^2c^2 + b^2c^2 + a^2b^2c^2 - \frac{1}{2}(1 + c^2) (a^2 + b^2 - a^2 b^2) - \\
& \qquad - \frac{1}{2}(1 + a^2)(b^2 + c^2 - 2 b^2 c^2)
\end{split} 
\\
\geq & \, 3a^2b^2c^2 + \frac{3}{2}a^2b^2 + \frac{1}{2}a^2 \frac{3}{2}b^2c^2 + \frac{1}{2}c^2 \\
\geq & \, 0.
\end{align*}
Hence $g(-1) \geq 0$ in this case. On the other hand, if $B \leq C$, then
\[ g(-1) = P - A + B - C \geq P - A - C. \]
Using lemma \ref{lem2}, we have
\begin{align*} & \, P - A - C \\
\begin{split}
\geq & \, a^2 + b^2 + c^2 + a^2b^2 + a^2c^2 + b^2c^2 + a^2b^2c^2 - \frac{1}{2}(1 + c^2)(a^2 + b^2 - a^2 b^2) - \\
& \qquad - \frac{1}{2}(1 - b^2)(a^2 + c^2 - 2 a^2 c^2)
\end{split} 
\\
\geq & \, a^2b^2c^2 + \frac{5}{2}a^2b^2 + \frac{3}{2}a^2c^2 + b^2c^2 \frac{1}{2}b^2 + \frac{1}{2}c^2 \\
\geq & \,0.
\end{align*}
Hence $g(-1) \geq 0$ in this case too.

We next check the endpoint $t = 1$. We have
\[ g(1) = P + A - B - C \geq P - B - C. \]
Using lemma \ref{lem2}, we have
\begin{align*} & \, P - B - C \\
\geq & \, a^2 + b^2 + c^2 + a^2b^2 + a^2c^2 + b^2c^2 + a^2b^2c^2 - \frac{1}{2}(1 + a^2)(b^2 + c^2 - 2 b^2 c^2) - \\
& \qquad -\frac{1}{2}(1 - b^2)(a^2 + c^2 - 2 a^2 c^2) \\
\geq & \, a^2b^2c^2 + a^2b^2 + \frac{3}{2} a^2c^2 + \frac{5}{2}b^2c^2 + \frac{1}{2}a^2 + \frac{1}{2}b^2 \\
\geq & \, 0.
\end{align*}
Hence $g(1) \geq 0$.

It remains to consider the potential critical points of $g$ which lie in $(-1, 1)$. Assume that $g$ has a critical point $t_c \in (-1, 1)$. Then
\[ g'(t_c) = 0, \]
so that
\[ A = \frac{BC}{\sqrt{B^2 + C^2 + 2BCt_c}}, \]
hence
\[ \sqrt{B^2 + C^2 + 2BCt_c} = \frac{BC}{A}. \]
We thus obtain
\[ g(t_c) = P - \frac{A^2B^2 + A^2C^2 + B^2C^2}{2ABC}. \]
Simplifying with the help of Julia's Symbolics package, we obtain
\[ g(t_c) = \frac{N}{2(1 + a^2)(1 + c^2)}, \]
where
\begin{equation*} 
\begin{split}
N & = a^2 + b^2 + c^2 + 3 a^4 + 3 a^2 b^2 + 2 a^2 c^2 + 3 b^2 c^2 + 3 c^4 + 5 a^4 c^2 + 8 a^2 b^2 c^2 + 5 a^2 c^4 + \\ 
& \qquad + a^4 b^2 c^2(1 - c^2) + 4 a^4 c^4 + a^2 b^2 c^4.
\end{split}
\end{equation*}
Since $c^2 \leq 1$, we can thus see that $N \geq 0$, from which we conclude that $g(t_c) \geq 0$, thus finishing the proof of Thm. \ref{main-thm}.

\section{Numerical Simulations}

For $n = 2$, Conjecture B is trivially true. For $3 \leq n \leq 6$, the following numerical simulation was run.

\begin{enumerate}
\item Generate the collection $\mathcal{G}_n$ of all non-oriented finite connected simple graphs with $n$ vertices, up to isomorphism.
\item For each $G \in \mathcal{G}_n$, generate (pseudo-)randomly $N = 10000$ configurations of $n$ distinct points in $\R^3$,  compute the corresponding $G$-amplitudes and store the results in a vector.
\item Check if the minimum real part of the stored amplitudes is greater or equal to $1$ minus a very small number, to allow for cases where the amplitude is precisely $1$, though numerically a tiny bit smaller than $1$, to go through. In case the check fails, this means we would have found a counterexample to the statement
\[ \Real(\mathcal{A}_G(\mathbf{x})) \geq 1 \]
and we can exit the program.
\end{enumerate}

It turns out that my program could not find any counterexample to the previous statement, for $n \leq 6$. It is thus tempting to conjecture that
\[ \Real(\mathcal{A}_G(\mathbf{x})) \geq 1 \]
for any $\mathbf{x} \in C_G(\R^3)$, assuming $G$ is connected. It is not true for general non-connected graphs. Indeed, it is enough to find a graph $G$ and a configuration $\mathbf{x} \in C_G(\R^3)$ with $\mathcal{A}_G(\mathbf{x})$ having non-zero imaginary part (there exist such a $(G, \mathbf{x})$). Let us say that the phase angle of $\mathcal{A}_G(\mathbf{x})$ is $\theta$, where $\theta \in (-\pi/2, \pi/2)$, $\theta \neq 0$. WLOG, we may assume that $0 < \theta < \pi/2$ (indeed, if this is not the case, we may replace $\mathbf{x}$ with $-\mathbf{x}$). Let $M$ be a positive integer such that $\pi/2 \leq M \theta \leq \pi$. Then replace $G$ with the disjoint union of $M$ copies of $G$ and replace $\mathbf{x}$ with $\mathbf{x}$ ``repeated'' $M$ times in total. Then, from the multiplicative property of $\mathcal{A}_G$, prop. \ref{multiplicative}, it follows that the corresponding amplitude is the original amplitude to the power $M$ and therefore definitely violates the inequality
\[ \Real(\mathcal{A}_G(\mathbf{x})) \geq 1. \]
The author wonders though if the latter inequality holds if $G$ is connected.

Our numerical simulations provide evidence for Conjectures B and C. We also ran some similar numerical simulations for small values of $n$ for our matrix analysis conjecture, namely Conjecture D and also could not find any counterexample.

It is tempting to conjecture that for any finite simple graph $G$, $c_G$ is attained at some collinear configuration, in analogy with the Atiyah--Sutcliffe setting, where this seems to be the case, based on numerical studies (cf. \cite{Ati-Sut-2002}). However, such a statement is false in general. For example, if $G$ is a ``star'' graph with $4$ vertices, i.e.
\begin{align*} V &= \{1, 2, 3, 4 \},\\
	E &= \{ \{1,2\}, \{1,3\}, \{1,4\} \},
\end{align*}
then the minimal $|\mathcal{A}_G|$ at a collinear configuration is $2$, while there is a configuration $\mathbf{x} \in C_G(\mathbb{R}^3)$ for which $|\mathcal{A}_G(\mathbf{x})| \simeq 1.5054511$). For that same graph $G$, which is a tree, the minimal $\operatorname{Re}(\mathcal{A}_G)$ at a collinear configuration is $2$, while there is a configuration $\mathbf{x}' \in C_G(\mathbb{R}^3)$ for which $\operatorname{Re}(\mathcal{A}_G(\mathbf{x}')) \simeq 1.5172676$.

\section{Amplitudes as tensor network contractions}

A tensor network is a graph $G = (V, E)$ together with the following data attached to $G$:
\begin{enumerate}
\item a vector space $\mathcal{H}_e$ attached to each edge $e \in E$,
\item a vector space $\mathcal{H}_v$ attached to each vertex $v \in V$,
such that
\[ \mathcal{H}_v = \bigotimes_{e \in N(v)} \mathcal{H}_e, \]
where $N(v)$ is the set of edges in $G$ that are incident with $v$ and
\item an element $T_v \in \mathcal{H}_v$, for every vertex $v \in V$.
\end{enumerate}
It is a graphical representation for the following construction. First, form
\[ \widetilde{T} = \otimes_{v \in V} T_v, \]
then, for each edge $e \in E$, say joining $v$ with $v'$ ($v, v' \in V$), contract the index of $T_v$ corresponding to $e$ with the index of $T_{v'}$ corresponding to $e$, then repeat this process till all edges in $E$ have been used once. At the end of the process, one gets a tensor, say $T$. The tensor network is just a graphical representation of the tensor $T$. One says that $T$ is the contraction of the tensor network.

Tensor networks have found applications in statistical physics, probability and statistics and machine learning. We shall not review tensor networks here (see instead the reviews \cite{Orus2014}, \cite{Orus2019}, \cite{BriChu2017} and the book \cite{Ran2020}, along with the references therein).

A closed tensor network is a tensor network without ``open-ended'' edges. It now should be clear that our $G$-amplitude $\mathcal{A}_G(\mathbf{x})$ is nothing but the contraction of a closed tensor network whose underlying graph is $G$, with the difference that contraction along each edge $e$ is done using the complex symplectic form  $\omega(-, -)$.

\section{Conclusion}

We have generalized the Atiyah problem on configurations of points using graph theory and tensors. We hope our work will prove helpful in finding a solution to the original problem. Either way, the conjectures we have presented are, in the author's humble opinion, interesting geometric inequalities.

While we were mainly concerned with minimizing $|\mathcal{A}_G|$ in this work, finding configurations that maximize $|\mathcal{A}_G|$ for a given graph $G$ is also worthy of further exploration. Indeed, configurations which maximize the absolute value of the Atiyah determinant, when plotted, look quite symmetric and seem to be related to fullerenes (see figure 3 in \cite{Ati-Sut-2002}).

One open problem is to try to express the optimal lower bounds $c_G$ of $\mathcal{A}_G$ in terms of known properties of the graph $G$. Even a conjectural formula for $c_G$ would be nice, as a first step.

We also wonder whether the $G$-amplitude functions $\mathcal{A}_G$ are linked to physics. The choice of terminology, ``amplitude'', is due to the existence of similarities with probability amplitudes in quantum physics (we note for instance that there are analogies between our setup and that of the Affleck--Kennedy--Lieb--Tasaki model in condensed matter physics).

We wish to investigate such questions further in the future.

\section{Declaration}

\emph{No datasets were generated or analysed during the current study, apart from the randomly generated configurations in the numerical simulations. The author has no competing interests to declare that are relevant to the content of this article.}

\section{Declaration of generative AI and AI-assisted technologies in the manuscript preparation process}

During the preparation of this work, the author has used ChatGPT, version 5.3, for part of the proof of Conjecture B for the linear graph with $n = 5$ vertices. He has also made use of an older version of ChatGPT to help him write a program to generate all finite graphs for a given number of vertices $n$.  After using this tool/service, the author reviewed and edited the content as needed and takes full responsibility for the content of the published article.

\end{document}